\documentclass{amsart}

\usepackage{epsfig}
\usepackage{latexsym}
\theoremstyle{plain}
\newtheorem{prop}{Proposition}
\newtheorem{thm}{Theorem}
\newtheorem{cor}{Corollary}

\newtheorem{lemma}{Lemma}

\usepackage[all]{xy}

\theoremstyle{remark}

\newtheorem{exm}{Example}

\usepackage{amssymb}
\usepackage{amsthm}
\usepackage{amsmath}
\usepackage{amsxtra}

\newcommand{\Affi}{\mathbb A}

\newcommand{\W}{W}

\newcommand{\la}{\langle}
\newcommand{\ra}{\rangle}
\newcommand{\rarr}{\rightarrow}

{.tfm}
{.tfm}

\begin{document}

\title[]{Solutions of a Linear Equation in a Subgroup of Units in a Function Field}

\author{Chia-Liang Sun}
\maketitle

\begin{itemize}
\item
   Institute of Mathematics,
   Academia Sinica\newline
   E-mail: {\tt csun@math.sinica.edu.tw}
\end{itemize}

\begin{abstract}
Over a large class of function fields, we show that the solutions of some linear equations in the topological closure of a certain  subgroup of the group of units in the function field are exactly the solutions that are already in the subgroup. This result solves some cases of the function field analog of an old conjecture proposed by Skolem.
\end{abstract}

\section{Introduction}\label{OnSkolemConj_intro}
Let $K/k$ be a function field, i.e., $K$ is finitely generated over $k$ with transcendence degree $1$ such that $k$ is relatively algebraically closed in $k$. Let $\Omega_{K/k}$ be the set of all places of $K/k$. Let $M$ be a natural number, and $\Affi^M$ be the affine $M$-space, whose coordinate is denoted by ${\mathbf X}=(X_1,\ldots,X_M)$. For any two elements $\mathbf{a}=(a_1,\ldots,a_M)$ and $\mathbf{b}=(b_1,\ldots,b_M)$ in $\Affi^M(K^*)$, we define $\mathbf{ab}=(a_1b_1,\ldots,a_Mb_M)\in\Affi^M(K)$ and $\mathbf{a}\cdot{\mathbf b}=\sum_{i=1}^Ma_ib_i\in K$; we also write $\mathbf{b}\cdot{\mathbf X}$ for the function $\sum_{i=1}^Mb_iX_i$ on $\Affi^M$. For any subset $\Theta$ of some ring and any variety $V$ in $\Affi^M$, let $V(\Theta)$ denote the set of points on $V$ with each coordinate in $\Theta$.
For any $\mathbf{b}\in \Affi^M(K^*)$, we denote by $\W_{\mathbf{b}}$ (resp. $\W'_{\mathbf{b}}$) the variety in $\Affi^M$ defined by $\mathbf{b}\cdot{\mathbf X} = 0$ (resp. $\mathbf{b}\cdot{\mathbf X} = 1$).
We fix a cofinite subset $\Omega\subset\Omega_{K/k}$ and  endow $\prod_{v\in\Omega}K_v^*$ with the natural product topology. Via the diagonal embedding, we identify any subgroup $\Gamma\subset K^*$ with its image in $\prod_{v\in\Omega}K_v^*$, and denote by $\overline{\Gamma}$ its topological closure. Moreover, for each $v\in\Omega_{K/k}$, the inclusion $\Gamma\subset K_v^*$ is continuous and therefore induces a subtopology of $\Gamma$, which  will be referred to as \em $v$-adic subtopology\em.

The purpose of this paper is investigate the circumstances where the equalities
\begin{eqnarray}
&\W_{\mathbf{b}}(\overline{\Gamma})  = \W_{\mathbf{b}}(\Gamma) & \label{maineqW0}\\
&\W'_{\mathbf{b}}(\overline{\Gamma})  = \W'_{\mathbf{b}}(\Gamma) & \label{maineqW1}
\end{eqnarray}
hold. In the  case where $M=1$, both sides of (\ref{maineqW0}) are always empty; however, Example 0 in \cite{Sun} shows that
(\ref{maineqW1}) may fail unless we make some restrictions on the \em largeness \em of $\Gamma$. In the characteristic zero case, this is indeed the only assumption needed.
\begin{thm}\label{OnSkolemConj_main0}
Suppose that $k$ has characteristic $0$. Let $\Gamma\subset K^*$ be a subgroup contained in $O_S^*$ for some finite $S\subset\Omega_{K/k}$. Then for any $v\in\Omega_{K/k}$, the $v$-adic subtopology of $\Gamma$ is discrete. In particular,  both (\ref{maineqW0}) and (\ref{maineqW1}) hold.
\end{thm}

In the case where  $k$ is finite, the main result in \cite{Sun} shows that (\ref{maineqW1}) holds if  $M=1$ and $\Gamma$ is finitely generated. Nevertheless, Example \ref{x_plus_y} suggests that even in the case where $M=2$, both (\ref{maineqW0}) and (\ref{maineqW1}) can fail in general, unless we put some mixed assumptions on $\mathbf{b}$ and $\Gamma$. Recall that $K$ is \em separably generated \em over $k$ if there exists $t\in K$ such that $K$ is finite separable over  $k(t)$. The separable Hilbert subset $H_k(f_1,\ldots,f_m; g)$  of $k$, where each $f_i(T,X)$ is a separable irreducible polynomial in $k(T)[X]$ and $g(T)$ is a nonzero polynomial in $k[T]$, consists of those  $a\in k$ such that $g(a)\neq 0$ and  each $f_i(a,X)$ is defined and irreducible in $k[X]$ \cite{FieldArith}.
For each $i$, let $\psi_i:\Affi^M(K^*)\rarr\Affi^M(K^*)$ be the map which replaces the $i$-th component of $\mathbf{a}\in\Affi^M(K^*)$ by $1$ and keeps the others unchanged. We shall prove the following result.
\begin{thm}\label{OnSkolemConj_main}
Suppose that $k$ has characteristic $p$, and either
\begin{itemize}
\item that $k$ is finite, or
\item that $K$ is separably generated  over  $k$, that each separable Hilbert subset of $k$ is infinite, and that $k$ contains only finitely many roots of unity.
\end{itemize}
Let $\mathbf{b}\in \Affi^M(K^*)$ be contained in $\Affi^M(O_S^*)$ for some finite $S\subset\Omega_{K/k}$, and  $\Gamma\subset O_S^*$ be a subgroup. For each natural number  $m$, let $R_m\subset \Gamma$ be a complete set of representatives of $\Gamma/\Gamma\cap (kK^{p^m})^*$.
\begin{enumerate}
\item Each $R_m$ is a finite set.\label{fin_con}
\item Suppose that there is some $m$ such that for every $\mathbf{r}\in \Affi^M(R_m)$ the components of $\mathbf{br}$ are linearly independent over $kK^{p^m}$. Then both sides of (\ref{maineqW0}) are empty.\label{W0}
\item Suppose that there is some $m$ such that for every $\mathbf{r}\in \Affi^M(R_m)$ the components of some $\psi_j(\mathbf{br})$ are linearly independent over $kK^{p^m}$. Then  (\ref{maineqW1}) holds, and the common set is finite with its cardinality no larger than  the number of $\mathbf{r}\in \Affi^M(R_m)$ such that the components of $\mathbf{br}$ and those of each $\psi_j(\mathbf{br})$ are linearly independent over $kK^{p^m}$. \label{W1}
\end{enumerate}
\end{thm}
A large class of fields satisfy the assumption of Theorem \ref{OnSkolemConj_main} on $k$. In fact, if $k$ is a global field, then each separable Hilbert subset of $k$ is infinite (\cite{FieldArith}, Theorem 13.3.5) and $k$ contains only finitely many roots of unity; these two properties are preserved under a purely transcendental extension with arbitrary cardinality as its transcendence degree ({\em{op. cit.}}, Proposition 13.2.1) and any algebraic extension with a finite separable degree ({\em{op. cit.}}, Proposition 12.3.3).

Some brief remarks on the hypotheses of Theorem \ref{OnSkolemConj_main} follow. In case where $M=1$, the assumptions in (\ref{W0}) and (\ref{W1})  are vacuous. When $k$ is finite and $M=2$, using Lemma 3 in \cite{axby}, we may replace the assumption in (\ref{W0}) with $b_1b_2^{-1}\notin\sqrt{\Gamma}$, and that in (\ref{W1}) with $\{b_1,b_2\}\not\subset\sqrt{\Gamma}$, where $\sqrt{G}=\{x\in K^*: x^n\in G \text{ for some }n\in\mathbb{N}\}$ for any subgroup $G\subset K^*$.

\begin{exm}\label{x_plus_y}
Let $K=\mathbb F_p(t)$ be a purely transcendental extension of $\mathbb F_p$, and $\Gamma=\la t, -t, 1-t\ra$ be the subgroup of $K^*$ generated by $t$, $-t$, and $1-t$.  Take $\Omega\subset\Omega_K$ to be a cofinite subset such that $\Gamma$ is contained in the valuation ring $O_v$ for every $v\in\Omega$, and $\mathbf{b}=(1,1)\in\Affi^2(K^*)$. The sequence
$(t^{p^{n!}})_{n\geq 1}$ in $\Gamma$ converges to $\alpha\in \overline{\Gamma}
\setminus K^*$ (\cite{Sun}, Example 1). Therefore we see that $(\alpha, -\alpha)\in \W_{\mathbf{b}}(\overline{\Gamma})\setminus\W_{\mathbf{b}}(\Gamma)$ and  $(\alpha, 1-\alpha)\in \W'_{\mathbf{b}}(\overline{\Gamma}) \setminus\W'_{\mathbf{b}}(\Gamma)$. \end{exm}

In case where $\Gamma\subset O_S^*$ and $\mathbf{b}\in\Affi^M(O_S^*)$ for some finite $S\subset\Omega_{K/k}$, the equalities (\ref{maineqW0}) and (\ref{maineqW1}) are stronger assertions than the function field analog of an old conjecture raised by Skolem \cite{Sko37}. To state his conjecture, we make the following abbreviation:
$$
\begin{array}{lll}
\mathfrak{S}_L(K/k,M,\mathbf{b}, S, \Gamma) & \text{stands for} &\text{For each nonzero ideal }I\subset O_S\text{ there exists }\\
&&\mathbf{x}\in\Affi^M(\Gamma)\text{ such that }\mathbf{b}\cdot{\mathbf x}\in I.\\
\mathfrak{S}_G(K/k,M,\mathbf{b}, S, \Gamma) & \text{stands for} &
\text{There exists } \mathbf{x}\in\Affi^M(\Gamma) \text{ such that } \mathbf{b}\cdot{\mathbf x}=0.\\
\end{array}
$$
In the case where $K$ is a number field and $S$ contains all Archimedean places, Skolem asserts that $\mathfrak{S}_L(K,M,\mathbf{b}, S, \Gamma)\Leftrightarrow\mathfrak{S}_G(K,M,\mathbf{b}, S, \Gamma)$ should always hold, and gives a proof when $M=2$ \cite{Sko37}. The validity of Skolem's assertion seems to have been largely ignored in the recent literature until Harari and Voloch \cite{HV} noticed its connection with (\ref{maineqW0}) and (\ref{maineqW1}). In fact, we may translate Theorem \ref{OnSkolemConj_main0} and Theorem \ref{OnSkolemConj_main} into results on Skolem's conjecture via the following equivalences:
$$
\begin{array}{cccccc}
\mathfrak{S}_L(K/k,M,\mathbf{b}, S, \Gamma) & \Leftrightarrow & \W_{\mathbf{b}}(\overline{\Gamma})\neq\emptyset & \Leftrightarrow &
\W'_{\phi_i(\mathbf{b})}(\overline{\Gamma})\neq\emptyset &\text{ where } \Omega=\Omega_{K/k}\setminus S\\
\mathfrak{S}_G(K/k,M,\mathbf{b}, S, \Gamma) & \Leftrightarrow &  \W_{\mathbf{b}}(\Gamma)\neq\emptyset & \Leftrightarrow &
\W'_{\phi_i(\mathbf{b})}(\Gamma)\neq\emptyset &
\end{array}
$$
for each $i$, where $\phi_i:\Affi^M(K^*)\rarr\Affi^{M-1}(K^*)$ is defined by
$$
(a_1,\ldots,a_M)\mapsto \left(-\frac{a_1}{a_i},\ldots,-\frac{a_{i-1}}{a_i},-\frac{a_{i+1}}{a_i},\ldots,-\frac{a_M}{a_i}\right).
$$
For instance, if $K/k$ satisfies the assumptions in Theorem \ref{OnSkolemConj_main}, then $\mathfrak{S}_L(K/k,2,\mathbf{b}, S, \Gamma)\Leftrightarrow\mathfrak{S}_G(K/k,2,\mathbf{b}, S, \Gamma)$ always holds.

\section{Proof of the Main Results}\label{pf}
\begin{lemma}\label{LN}
Let $\Gamma\subset K^*$ be a subgroup contained in $O_S^*$ for some finite $S\subset\Omega_{K/k}$. Then for any $v\in\Omega_{K/k}$, there is a subgroup $G_v\subset\Gamma$ which is open in the $v$-adic subtopology. If we further suppose that $k$ contains only finitely many roots of unity, then $G_v$ may be chosen  such that $\sqrt{G_v}$ is finitely generated.
\end{lemma}
\begin{proof}
Taking $G_v=\Gamma\cap(1+m_v)$, we see that $G_v$ is open in the $v$-adic subtopology of  $\Gamma$ since $1+m_v$ is an open subgroup of $K_v^*$. The inclusion $\Gamma\subset O_S^*$ induces the map $G_v\rarr O_S^*/k^*$, which is injective because $k^*\cap(1+m_v)$ is trivial. The first conclusion follows from the fact that $O_S^*/k^*$ is  finitely generated (Corollary 1 of Proposition 14.1, \cite{NTFF}). Now we show that $\sqrt{G_v}$ is finitely generated under the additional hypothesis that $k$ contains only finitely many roots of unity.
Note that $\sqrt{G_v}\subset\sqrt{\Gamma}\subset O_S^*$, which again induces the map $\sqrt{G_v}\rarr O_S^*/k^*$ with its kernel  contained in $k^*\cap\sqrt{K^*\cap(1+m_v)}$, which is exactly the group of roots of unity in $k$. This completes the proof.
\end{proof}

{\it Proof of Theorem \ref{OnSkolemConj_main0}}:
Fix $v\in\Omega_{K/k}$ and let $U_n=\Gamma\cap(1+m_v^n)$ for $n\geq 1$. Then $U_n$ is open in the $v$-adic subtopology of $\Gamma$. Since $k$ has characteristic zero, the quotient groups $U_n/U_{n+1}$ are torsion-free for all $n$. The proof of Lemma \ref{LN} shows that $U_1$ is finitely generated, hence $U_n$ is trivial for some $n$.\qed

\begin{lemma}\label{remain_prime}
Suppose that $K$ is separably generated  over $k$, and that each separable Hilbert subset of $k$ is infinite. Let $L$ be a finite separable extension of $K$. Then there are infinitely many $v\in\Omega_{K/k}$ which extend uniquely and unramifiedly to a place of $L$.
\end{lemma}
\begin{proof}
Since $K$ is separably generated  over $k$, it is enough to assume that $K=k(t)$, in which case the desired property follows from the infiniteness of the separable Hilbert subset $H_k(f; 1)$, where $L=K(y)$ with $y$ a root of the separable polynomial $f\in k(T)[X]$.
\end{proof}

\begin{lemma}\label{ptop}
Suppose that $K/k$ satisfies the assumptions in Theorem \ref{OnSkolemConj_main}.
Let $\Gamma\subset K^*$ be a subgroup contained in $O_S^*$ for some finite $S\subset\Omega_{K/k}$. Then for any integer $m$ prime to $p$, the subgroup $\Gamma^m$ of $\Gamma$ is open.
\end{lemma}
\begin{proof}
The case where $k$ is finite is proved in Lemma 12 of \cite{Sun}. Thus we assume that $k$ is infinite.
By Lemma \ref{LN}, we may assume that $\sqrt{\Gamma}$ is  finitely generated; then by Lemma 5 of \cite{Sun}, we may further assume that $\Gamma=\sqrt{\Gamma}$. Let $L$ be the finite Galois extension of $K$ obtained by adjoining all the $m$-th roots of every element in $\Gamma$. For each field $E$ such that $K\subset E \subset L$, Lemma \ref{remain_prime} yields a place $v_E\in\Omega$ which extends to a unique place $v_E$ of $E$ such that $[E_{v_E}:K_{v_E}]=[E:K]$. Let $S$ be the finite set consisting of those $v_E$ such that there is no proper intermediate field between $K$ and $E$. We shall
complete the proof by showing that $\Gamma\cap U_S\subset \Gamma^m$, where $U_S=\prod_{v\in S}1+m_v$ is an open subgroup of $\prod_{v\in S}K_v^*$. In fact, we only have to show that $\Gamma\cap U_S\subset (K^*)^m$ because $\Gamma=\sqrt{\Gamma}$. Assume $x\in\Gamma\cap U_S\setminus (K^*)^m$ and let $F$ be the extension of $K$ obtained by adjoining an $m$-th root of $x$. Then we have $K\subsetneq E\subset F \subset L$ for some $E$ such that $v_E\in S$. Since $[E_{v_E}:K_{v_E}]=[E:K]\neq 1$, it follows that $x$ has no $m$-th root in $K_{v_E}$, which contradicts the assumption $x\in U_S$ by Hensel's lemma.
\end{proof}

\begin{lemma}
Suppose that $k$ has characteristic $p$, and that $k$ contains only finitely many roots of unity.
Let $\Gamma\subset K^*$ be a subgroup contained in $O_S^*$ for some finite $S\subset\Omega_{K/k}$.
Then for every $v\in\Omega_{K/k}$, any subgroup of $\Gamma$ containing  $\Gamma^{p^n}$ for some $n\in\mathbb{N}$ is open in the $v$-adic subtopology of $\Gamma$.
\end{lemma}
\begin{proof}
It suffices to show that those subgroups $\Gamma^{p^n}$ are open in the $v$-adic subtopology of $\Gamma$. As in the proof of Lemma \ref{ptop}, we may assume that $\Gamma=\sqrt{\Gamma}$ is  finitely generated. Because $(K_v^*)^{p^n}\cap K^*= (K^*)^{p^n}$, we have $(K_v^*)^{p^n}\cap\Gamma\leq (K^*)^{p^n}\cap\Gamma=\Gamma^{p^n}$.
Then it suffices to show that $(K_v^*)^{p^n}\cap\Gamma$ is open in the $v$-adic subtopology of $\Gamma$. Note that $(K_v^*)^{p^n}$ is closed in $K_v^*$ and consequently $K_v^*/(K_v^*)^{p^n}$ is Hausdorff. Consider the map $\Gamma\rarr K_v^*/(K_v^*)^{p^n}$ induced from the inclusion $\Gamma\subset K_v^*$, which is continuous with respect to the $v$-adic subtopology of $\Gamma$. Since this map factors through $\Gamma/\Gamma^{p^n}$, its image is finite, whence discrete. This completes our proof.
\end{proof}

\begin{cor}\label{csp}
Suppose that $K/k$ satisfies the assumptions in Theorem \ref{OnSkolemConj_main}.
Let $\Gamma\subset K^*$ be a subgroup contained in $O_S^*$ for some finite $S\subset\Omega_{K/k}$. Then any subgroup of $\Gamma$ containing  $\Gamma^m$ for some $m\in\mathbb{N}$ is open.\qed
\end{cor}

\begin{cor}\label{closed}
Suppose that $K/k$ satisfies the assumptions in Theorem \ref{OnSkolemConj_main}. Let $\Gamma\subset K^*$ be a subgroup contained in $O_S^*$ for some finite $S\subset\Omega_{K/k}$. Then $\Gamma$ is closed in $K^*$.
\end{cor}
\begin{proof}
Let $P\in\overline{\Gamma}\cap K^*$.  Lemma \ref{LN} (and its proof) shows that $P\in\overline{\Gamma_0}$ for some finitely generated subgroup $\Gamma_0\subset\Gamma$ such that $\Gamma_0$ is contained in a finitely generated closed subgroup $\Gamma_S\subset O_S^*$. By enlarging $S$, we may assume $S\cup\Omega=\Omega_{K/k}$ and hence both $O_S^*$ and $\Gamma_S$ are closed in $K^*$. By Corollary \ref{csp}, every subgroup of $\Gamma_S$ with finite index is open; hence $\Gamma_0$ is closed in $\Gamma_S$, and thus in $K^*$. This shows $P\in\Gamma_0\subset\Gamma$ and finishes our proof.
\end{proof}

\begin{cor}\label{iso}
Suppose that $K/k$ satisfies the assumptions in Theorem \ref{OnSkolemConj_main}. Let $\Gamma\subset K^*$ be a subgroup contained in $O_S^*$ for some finite $S\subset\Omega_{K/k}$. Then for any subgroup $\Delta$ of $\Gamma$ containing  $\Gamma^m$ for some $m\in\mathbb{N}$,  the homomorphism
$$
\Gamma/\Delta\rarr\overline{\Gamma}/\overline{\Delta}
$$
is bijective.
\end{cor}
\begin{proof}
By Corollary \ref{closed}, we have $\Gamma\cap\overline{\Delta}=\Gamma\cap (\overline{\Delta}\cap K^*)=\Delta$, which is the desired injectivity. By Corollary \ref{csp}, $\Delta$ is open in $\Gamma$, hence the desired surjectivity follows. (c.f. Lemma 8, \cite{Sun})
\end{proof}

\begin{lemma}\label{temp}
Suppose that $k$ has positive characteristic $p$, and that $K$ is separably generated over  $k$. Then for any $n>0$, we have  $K\cap\overline{k}K^{p^n}=kK^{p^n}$.
\end{lemma}
\begin{proof}
Denote by $k^{\text{sep}}$ the separable closure of $k$ in $\overline{k}$. As $K\cap k^{\text{sep}}K^{p^n}$ is both separable and purely inseparable over $kK^{p^n}$, they are equal to each other; hence it remains to show that $K\cap\overline{k}K^{p^n}\subset k^{\text{sep}}K^{p^n}$. Let $K=k(t,y)$ with $y$ separable over $k(t)$. Then $K=k(t,y^{p^n})$ and thus $K\cap \overline{k}K^{p^n}=k(t,y^{p^n})\cap \overline{k}(t^{p^n},y^{p^n})\subset k^{\text{sep}}(t,y^{p^n})\cap \overline{k}(t^{p^n},y^{p^n})$. The irreducible polynomial of $y^{p^n}$ over $k^{\text{sep}}(t)$ is still irreducible over $\overline{k}(t)$,  and therefore $k^{\text{sep}}(t,y^{p^n})\cap \overline{k}(t^{p^n},y^{p^n}) =k^{\text{sep}}(t^{p^n},y^{p^n})$ since $k^{\text{sep}}(t)\cap \overline{k}(t^{p^n})=k^{\text{sep}}(t^{p^n})$. This completes the proof.
\end{proof}

\begin{lemma}\label{cts}
Suppose that $k$ has positive characteristic $p$, and that $K$ is separably generated over  $k$. Then for each $v\in\Omega_{K/k}$ and each $n>0$, any $kK^{p^n}$-linear map $\phi: K\rarr K$ is continuous with respect to the $v$-adic topology.
\end{lemma}
\begin{proof}
Any $v\in\Omega_{K/k}$ gives a discrete valuation $v:K^*\rarr\mathbb Z$ such that $v(s)=1$ for some $s\in K^*$ and that $v((kK^{p^n})^*)\subset p^n\mathbb Z$. It follows that $\{s^j\}_{0\leq j\leq p^n-1}$ is $K^{p^n}$-linearly independent. Since $K$ is separably generated over  $k$, we have $[K:kK^{p^n}]=p^n$ and conclude that
$\{s^j\}_{0\leq j\leq p^n-1}$ is a $K^{p^n}$-linear basis for $K$. To prove this lemma, it is enough to show the continuity of $\phi$ at $0$. Let $x = \sum_{j=0}^{p^n-1} c_j s^j\neq 0$ with all $c_j\in kK^{p^n}$. Then $v(x) = \min_{c_j\neq 0} \left(v(c_j)+j\right)$ and $\phi(x) = \sum_{j=0}^{p^n-1} c_j \phi(s^j)$. Thus, we have $v(\phi(x)) \geq \min_{c_j\neq 0} \left( v(c_j)+v(\phi(s^j))\right)\geq v(x)+ \min_{0\leq j\leq p^n-1} \left(v(\phi(s^j)-j)\right)$. This finishes the proof since $\min_{0\leq j\leq p^n-1} \left(v(\phi(s^j)-j)\right)$ is independent of $x$.
\end{proof}

For each $v\in\Omega_{K/k}$ and each natural number $m$, denote by $(kK^{p^m})^*_v$ the topological closure of $(kK^{p^m})^*$ in $K_v^*$.

\begin{prop}\label{OnSkolemConj_key}
Suppose that $k$ has positive characteristic $p$, and that $K$ is separably generated over  $k$. Let $\mathbf{b}\in (K^*)^M$, and let $m$ be a natural number.
\begin{enumerate}
\renewcommand{\labelenumi}{\theenumi)}
\renewcommand{\theenumi}{\alph{enumi}}
\item Suppose that the components of $\mathbf{b}$ are linearly independent over $kK^{p^m}$. Then we have $$\W_{\mathbf{b}}\left((kK^{p^m})^*_v\right)=\emptyset\qquad\text{for all }v\in\Omega_{K/k}.$$\label{keyW0}
\item Suppose that for some $j$ the components of $\psi_j(\mathbf{b})$  are linearly independent over $kK^{p^m}$. Then for some $P\in\W'_{\mathbf{b}}(K)$ we have
    $$
    \prod_{v\in\Omega_{K/k}}\W'_{\mathbf{b}}\left((kK^{p^m})^*_v\right)\subset \{P\}.
    $$
    If, moreover,  the components of either $\mathbf{b}$ or some $\psi_l(\mathbf{b})$  are linearly dependent over $kK^{p^m}$, then we have $$\W'_{\mathbf{b}}\left((kK^{p^m})^*_v\right)=\emptyset\qquad\text{for all }v\in\Omega_{K/k}.$$ \label{keyW1}
\end{enumerate}
\end{prop}
\begin{proof}
Choose $t\in\overline{k}K$ such that $\overline{k}(t)\subset \overline{k}K \subset \overline{k}((t))$.
By Remark 1 in \cite{VolochWronskians}, there exists an iterative derivation $\{D_{\overline{k}K}^{(i)}\}_{i\geq 0}$ on $\overline{k}K$  such that $\overline{k}K^{p^m}=\{x\in \overline{k}K:D_{\overline{k}K}^{(l)}(x)=0,\;\text{ if }\; 1\leq l< p^m\}$ for $i,j\geq 0$. Taking restriction gives an iterative derivation $\{D_K^{(i)}\}_{i\geq 0}$ on $K$ such that $\{x\in K :D_K^{(l)}(x)=0,\;\text{ if }\; 1\leq l< p^m\}=K\cap \overline{k}K^{p^m} =kK^{p^m}$ by Lemma \ref{temp}.  By Lemma \ref{cts}, for each $v\in\Omega_{K/k}$, we extend each $\{D_K^{(i)}\}_{i\geq 0}$ to an iterative derivation $\{D_{K_v}^{(i)}\}_{i\geq 0}$ on $K_v$ by continuity, and note that $D_{K_v}^{(i)}|_{\left((kK^{p^m})^*\right)_v}$ is the zero map for any $1\leq i<p^m$.

Fix some $v\in\Omega_{K/k}$ and some $\mathbf{c}=(c_1,\ldots, c_M)\in \W_{\mathbf{b}}\left((kK^{p^m})^*_v\right)\cup \W'_{\mathbf{b}}\left((kK^{p^m})^*_v\right)$.  Then   $\sum_{j=1}^N b_j c_j = e$, where $e\in\{0,1\}$. For any $0\leq i<p^m$, because $D_{K_v}^{(i)}(c_j)=0$, we have that
\begin{equation}\label{keyeq}
\sum_{j=1}^N D_K^{(i)}(b_j) c_j^{(l)}=D_K^{(i)}(e).
\end{equation}

Denote by $\mathbf{c}$ (resp. $\mathbf{e}$) the $M$-by-$1$ matrix with the $j$-th component $c_j$ (resp. $D_K^{(j)}(e)$).
For any set $I=\{i_1,\ldots,i_M\}$ of $M$ nonnegative integers such that $0=i_1<i_2<\cdots i_M<p^m$, let $\mathbf{T}_{\mathbf{b},I}$ be the $M$-by-$M$ matrix with the entry $D_K^{(i_l)}(b_j)$ being at the $l$-th row and $j$-th column. From (\ref{keyeq}), we have $\mathbf{T}_{\mathbf{b},I}\mathbf{c}=\mathbf{e}$, which implies
\begin{equation}\label{mateq}
(\det \mathbf{T}_{\mathbf{b},I}) \mathbf{c}=\mathbf{T}_{\mathbf{b},I}^*\mathbf{e},
\end{equation}
where $\mathbf{T}_{\mathbf{b},I}^*$ denotes the adjoint matrix of $\mathbf{T}_{\mathbf{b},I}$. If $e=0$ and the components of $\mathbf{b}$ are linearly independent over $kK^{p^m}$, then by Theorem 1 of \cite{VolochWronskians}, $\det \mathbf{T}_{\mathbf{b},I}\neq 0$ for some $I$, which contradicts ${\mathbf c}\neq 0$. This proves (\ref{keyW0}).

To prove (\ref{keyW1}), we consider the case where $e=1$. Note that the $j$-th component of $\mathbf{T}_{\mathbf{b},I}^*{\mathbf e}$ is exactly $\det\mathbf{T}_{\psi_j(\mathbf{b}),I}$. Under the assumption in the first part, Theorem 1 of \cite{VolochWronskians} implies that $\det \mathbf{T}_{\psi_j(\mathbf{b}),I}\neq 0$ for some $j$ and $I$. Hence there is at most one choice for ${\mathbf c}$ satisfying (\ref{mateq}), and this choice gives $P\in W'_{\mathbf{b}}(K)$.  If the additional hypothesis also holds, then either $\det \mathbf{T}_{\mathbf{b},I}$ or some component of $\mathbf{T}_{\mathbf{b},I}^*{\mathbf e}$ is zero, and (\ref{mateq}) is impossible.
\end{proof}

{\it Proof of Theorem \ref{OnSkolemConj_main}.}
First, note that for each $m$ the kernel of the natural map
$$
\Gamma\rarr\left(O_S^*/k^*\right)/\left(O_S^*/k^*\right)^m
$$
is contained in $(kK^{p^m})^*$. This proves (\ref{fin_con}) since $O_S^*/k^*$ is  finitely generated (Corollary 1 of Proposition 14.1, \cite{NTFF}). We have
$$
\overline{\Gamma}=\bigcup_{\gamma\in R_m} \gamma\overline{\Gamma\cap (kK^{p^m})^*}\subset
\bigcup_{\gamma\in R_m} \prod_{v\in\Omega}\gamma (kK^{p^m})^*_v,
$$
where the first equality follows from Corollary \ref{iso}. This gives
$$
\W_{\mathbf{b}}(\overline{\Gamma})\subset
\prod_{v\in\Omega} \W_{\mathbf{b}}\left(\bigcup_{\gamma\in R_m} \gamma (kK^{p^m})^*_v\right)
=\bigcup_{\mathbf{r}\in\Affi^M(R_m)}\prod_{v\in\Omega}\mathbf{r} \W_{\mathbf{br}}\left((kK^{p^m})^*_v\right).
$$
Proposition \ref{OnSkolemConj_key}(\ref{keyW0}) shows that $\prod_{v\in\Omega}\W_{\mathbf{br}}\left((kK^{p^m})^*_v\right)=\emptyset$ for each $\mathbf{r}\in\Affi^M(R_m)$,  proving (\ref{W0}).

Similarly, we have
$$
\W'_{\mathbf{b}}(\overline{\Gamma})\subset
\bigcup_{\mathbf{r}\in\Affi^M(R_m)}\prod_{v\in\Omega}\mathbf{r} \W'_{\mathbf{br}}\left((kK^{p^m})^*_v\right).
$$
Let $U\subset \Affi^M(R_m)$ be the subset consisting of those $\mathbf{r}$ such that the components of $\mathbf{br}$ and those of each $\psi_j(\mathbf{br})$ are linearly independent over $kK^{p^m}$.   By Proposition \ref{OnSkolemConj_key}(\ref{keyW1}),  for each $\mathbf{u}\in U$ there exists some $P_{\mathbf{u}}\in\W'_{\mathbf{bu}}(K)$ such that $\prod_{v\in\Omega}\W'_{\mathbf{br}}\left((kK^{p^m})^*_v\right)\subset\{P_{\mathbf{u}}\}$, while for each $\mathbf{r}\in R^M\setminus \Affi^M(R_m)$, we have $\prod_{v\in\Omega}\W'_{\mathbf{br}}\left((kK^{p^m})^*_v\right)=\emptyset$. Then,
$\W'_{\mathbf{b}}(\overline{\Gamma})
\subset\{\mathbf{u}P_{\mathbf{u}}: \mathbf{u}\in U\}\subset \W'_{\mathbf{b}}(K)$. It follows that $\W'_{\mathbf{b}}(\overline{\Gamma})
\subset\W'_{\mathbf{b}}(\overline{\Gamma})\cap \W'_{\mathbf{b}}(K)=\W'_{\mathbf{b}}(\overline{\Gamma}\cap K^*)= \W'_{\mathbf{b}}(\Gamma)$, where the last equality is concluded by Corollary \ref{closed}. This finishes our proof.\qed
\bibliographystyle{alpha}
\bibliography{mybib}
\end{document}